\documentclass[10pt,a4paper]{article}
\usepackage{latexsym,amsthm,amssymb,amscd,amsmath}
\newcounter{alphthm}
\theoremstyle{plain}
\newtheorem{theorem}{Theorem}[section]

\newtheorem{proposition}[theorem]{Proposition}
\newtheorem{cor}[theorem]{Corollary}
\theoremstyle{definition}
\newtheorem{defn}[theorem]{Definition}
\newtheorem{rem}[theorem]{Remark}
\newtheorem{example}[theorem]{Example}

\newcommand{\be}{\begin{equation}}
\newcommand{\ee}{\end{equation}}
\newcommand{\ben}{\begin{enumerate}}
\newcommand{\een}{\end{enumerate}}

\begin{document}
\title{On $\phi$-Recurrent Contact Metric Manifolds}
\author{E. Peyghan, H. Nasrabadi and A. Tayebi}
\maketitle

\maketitle
\begin{abstract}
In this paper, we prove that evry 3-dimensional manifold $M$ is a $\phi$-recurrent $N(k)$-contact metric manifold if and only if it is flat. Then we classify the $\phi$-recurrent contact metric manifolds of constant curvature. This implies that there exists no $\phi$-recurrent $N(k)$-contact metric manifold, which is neither symmetric nor locally $\phi$-symmetric.
\end{abstract}

\section{Introduction}
In 1872, S. Lie  introduced the notion of contact transformation as a geometric tool to study systems of differential equations \cite{A1}\cite{A2}\cite{B}. The theory of contact metric structures occupies one of the leading places in researches of modern differential geometry because of its several applications in mechanics, optics, phase space of a dynamical system, control theory  and in the theory of geometrical quantization \cite{G}.

On the other hand, the internal contents of the theory of contact metric structures are rich and have close substantial interactions with other parts of geometry. For example, Sasakian manifolds play important role in contact geometry. Indeed, the links between contact geometry and complex geometry are especially strong for Sasakian manifolds \cite{B}\cite{BC}\cite{C}.

It is shown that, the only locally symmetric Sasakian manifolds are
locally isometric to $S^{2n+1}(1)$ and that the only locally symmetric contact
metric manifolds are locally isometric to $S^{2n+1}(1)$ or to $E^{n+1}\times S^n(4)$ (see \cite{B}).
Certainly this can be regarded as saying that the idea of being locally
symmetric is too strong. For this reason, this notion has been weakend
by many geometers in different ways such as recurrent manifold by Walker \cite{W}, semi
symmetric manifold by Szab\'{o} \cite{Sz}, pseudo-symmetric manifold by Chaki \cite{Ch}, and
Deszcz \cite{De} and  weakly symmetric manifold by Tammasy and Binh \cite{TB}, and
Selberg \cite{Se}. As a weaker version of local symmetry, Takahashi introduced
the notion of a locally $\phi$-symmetric space \cite{T}. Generalizing the notion of local $\phi$-symmetry, De-Shaikh-Biswas  introduced the notion of $\phi$-recurrent Sasakian manifold \cite{DSB}. Then in \cite{DG}, De-Gazi studied $\phi$-recurrent $N(k)$-contact metric manifolds and generalized the results of \cite{DSB}.

In \cite{DG}, De-Gazi  proved that a 3-dimensional $\phi$-recurrent $N(k)$-contact metric manifold
is of constant curvature. Then they provided the existence of the
$\phi$-recurrent $N(k)$-contact metric manifold by means of an example which is neither
symmetric nor locally $\phi$-symmetric.

This paper is arranged as follows. In Section 2, we present some basic concepts of the contact metric manifolds, Sasakian manifolds, locally $\phi$-symmetric manifolds and $\phi$-recurrent $N(k)$-contact metric manifolds. In Section 3, we show that the example introduced by De-Gazi in \cite{DG} is not correct. Then, we prove that a 3-dimensional manifold $M$ is $\phi$-recurrent $N(k)$-contact metric manifold if and only if it is a flat manifold. In other words, we prove that there exists no 3-dimensional $\phi$-recurrent $N(k)$-contact metric manifold for $k\neq 0$. We also deduce that there exists no  3-dimensional $\phi$-recurrent $N(k)$-contact metric manifold which is neither symmetric nor locally $\phi$-symmetric. All results in this section show that Theorem 4.1 in \cite{DG} is not correct. In Section 4, we show that for $k\neq 0$, there is no $(2n+1)$-dimensional $\phi$-recurrent $N(k)$-contact metric manifold of contact curvature. We also prove that there is no $(2n+1)$-dimensional $\phi$-recurrent contact metric manifold of contact curvature for $n>1$. Finally we show that  only, the flat 3-dimensional manifolds are $\phi$-recurrent contact metric manifold of constant curvature.
\section{Contact Metric Manifolds}
In this section, we remark some fundamental materials about contact metric geometry. We refer to \cite{B}, \cite{L} for further details.

A $(2n+1)$-dimensional manifold $M^{2n+1}$ is said to be a contact manifold if it admits a global 1-form $\eta$ such that $\eta\wedge(d\eta)^n\neq0$, everywhere. Given a contact form $\eta$, there exists a unique vector field $\xi$, the characteristic vector field, which satisfies $\eta(\xi)=1$ and $d\eta(\xi, X)=0$ for any vector field $X$. It is well known that,  there exists an associated Riemannian metric $g$ and a (1, 1)-type tensor field $\phi$ such that the following hold
\begin{equation}\label{con}
(i)\ \eta(X)=g(X, \xi),\ \ \ (ii)\ d\eta(X, Y)=g(X, \phi Y),\ \ \ (iii)\ \phi^2X=-X+\eta(X)\xi,
\end{equation}
where $X$ and $Y$ are vector fields on $M$. By (\ref{con}), it follows that
\[
\phi\xi=0,\ \ \ \eta\circ\phi=0,\ \ \ g(\phi X, \phi Y)=g(X, Y)-\eta(X)\eta(Y).
\]
A Riemannian manifold $M$ equipped with the structure tensors $(\eta, \xi, \phi, g)$ satisfying (\ref{con}) is said to be a
\textit{contact metric manifold}.

Given a contact metric manifold $M$, we define a (1, 1)-tensor field $h$ by $h=\frac{1}{2}\pounds_\xi\phi$, where $\pounds$ denotes the Lie differentiation. Then the tensor $h$ is symmetric and satisfies
\begin{align}
(i)\ h\xi&=0,\ \ (ii)\ h\phi+\phi h=0,\ \ (iii)\ \nabla_X\xi=-\phi X-\phi hX,\label{Kill}\\
2(\nabla_{hX}\phi)Y&=-R(\xi, X)Y-\phi R(\xi, X)\phi Y+\phi R(\xi, \phi X)Y-R(\xi, \phi X)\phi Y\nonumber\\
&\ \ \ \ +2g(X+hX, Y)\xi-2\eta(Y)(X+hX)\nonumber,
\end{align}
where $\nabla$ is the Levi-Civita connection and $R$ is the Riemannian curvature tensor of $M$ defined
by following
\[
R(X, Y)Z=\nabla_X\nabla_YZ-\nabla_Y\nabla_XZ-\nabla_{[X, Y]}Z,\ \ \ \forall X, Y, Z\in\chi(M).
\]
For a contact metric manifold $M$,  one may defines naturally an almost complex structure $J$ on $M\times\mathbb{R}$ as follows
\[
J(X, f\frac{d}{dt})=(\phi X-f\xi, \eta(X)\frac{d}{dt}),
\]
where $X$ is a vector field tangent to $M$, $t$ the coordinate on $\mathbb{R}$ and $f$ a function on $M\times\mathbb{R}$. If the almost
complex structure $J$ is integrable, $M$ is said to be \textit{normal or Sasakian}. It is known that,  a contact metric
manifold $M$ is normal if and only if $M$ satisfies
\[
 [\phi, \phi]+2d\eta\otimes\xi=0,
\]
where $[\phi, \phi]$ is the Nijenhuis torsion of $\phi$. It is also well known that a contact metric manifold $M$ is
Sasakian if and only if
\[
R(X, Y)\xi=\eta(Y)X-\eta(X)Y,\ \ \ \forall X, Y\in\chi(M).
\]

The $k$-nullity distribution $N(k)$ of a Riemannian manifold $M$ is defined by
\[
N(k):p\longrightarrow N_p(k)=\Big\{Z\in T_pM: R(X, Y)Z=k[g(Y, Z)X-g(X, Z)Y]\Big\},
\]
where $k$ is a constant. If the characteristic vector field $\xi$ belongs to $N(k)$, then we call a contact metric manifold an $N(k)$-contact metric manifold. If $k=1$,
then $N(k)$-contact metric manifold is Sasakian and if $k=0$, then $N(k)$-contact metric manifold is locally isometric to the product $E^{n+1}\times S^n(4)$ for
$n>1$ and flat for $n=1$ (see \cite{B}). For a $N(k)$-contact metric manifold we have
\begin{equation}\label{esi3}
(i)\ R(X, Y)\xi=k[\eta(Y)X-\eta(X)Y],\ \ \  \ \ (ii)\ S(X, \xi)=2nk\eta(X),
\end{equation}
where $S$ is the Ricci tensor of Riemannian manifold $(M, g)$  (see \cite{DG}).

A contact metric manifold is said to be \textit{locally $\phi$-symmetric} if the relation
\[
\phi^2((\nabla_WR)(X, Y)Z)=0,
\]
holds for all vector fields $X$, $Y$, $Z$, $W$ orthogonal to $\xi$ \cite{BKS}. This notion was introduced for Sasakian manifolds by Takahashi \cite{T}.
\begin{defn}\label{hasan6}
A contact metric manifold is said to be $\phi$-recurrent if there exists a non-zero 1-form $A$ such that
\begin{equation}\label{rec}
\phi^2((\nabla_WR)(X, Y)Z)=A(W)R(X, Y)Z,
\end{equation}
for all vector fields $X, Y, Z, W$. 
\end{defn}
 In the above definition, $X, Y, Z, W$ are arbitrary vector fields and not necessarily orthogonal to $\xi$. This notion was introduced for Sasakian manifolds by De, Shaikh and Biswas \cite{DSB} and was introduced for $N(k)$-contact manifolds by De and Gazion \cite{DG}.
\begin{rem}\label{rem}
Flat manifolds are trivial examples of $\phi$-recurrent contact metric manifolds (locally $\phi$-symmetric manifolds), because for a flat manifold we have $R=0$ and $\nabla R=0$.
\end{rem}
\section{3-dimensional $\phi$-recurrent $N(k)$-contact metric  manifolds}
In \cite{DG}, De-Gazi presented the following example of $\phi$-recurrent $N(k)$-contact metric
manifold which is neither symmetric nor locally $\phi$-symmetric.
\begin{example}\label{Ex}
We take the 3-dimensional manifold $M=\{(x, y, z)| x\neq 0\}$, where $(x, y, z)$ are the
standard coordinates in $\mathbb{R}^3$. Let $E_1$, $E_2$, $E_3$ be linearly independent global frame on $M$ given by
\[
E_1=\frac{2}{x}\frac{\partial}{\partial y},\ \ E_2=2\frac{\partial}{\partial x}-\frac{4z}{x}\frac{\partial}{\partial y}+xy\frac{\partial}{\partial z},\ \ E_3=\frac{\partial}{\partial z}.
\]
Let $g$ be the Riemannian metric defined by $g(E_1, E_3)=g(E_2, E_3)=g(E_1, E_2)=0$ and $g(E_1, E_1)=g(E_2, E_2)=g(E_3, E_3)=0$. Let $\eta$ be the 1-form defined by $\eta(U)=g(U, E3)$ for any $U\in\chi(M)$. Let $\phi$ be
the (1, 1) tensor field defined by $\phi E_1=E_2$, $\phi E_2=-E_1$, $\phi E_3=0$. Then
using the linearity of $\phi$ and $g$ we have $\eta(E3)=1$, $\phi^2(U)=-U+\eta(U)E_3$ and $g(\phi U, \phi W)=g(U, W)-\eta(U)\eta(W)$ for any $U, W\in\chi(M)$. Moreover
$hE_1=-E_1$, $hE_2=E_2$, $hE_3=0$. Thus for $E_3=\xi$, $(\phi, \xi, \eta, g)$ defines
a contact metric structure on $M$. Hence we have $[E_1, E_2]=2E_3+\frac{2}{x}E_1$, $[E_1, E_3]=0$, $[E_2, E_3]=2E_1$.

The Riemannian connection $\nabla$ of the metric $g$ is given by
\begin{align*}
2g(\nabla_XY, Z)&=Xg(Y, Z)+Yg(Z, X)-Zg(X, Y)\\
&\ \ \ -g(X, [Y, Z])-g(Y, [X, Z])+g(Z, [X, Y]).
\end{align*}
Taking $E_3=\xi$ and using the above formula for Riemannian metric $g$, it can
be easily calculated that
\begin{align}\label{hasan}
\nabla_{E_1}E_3&=0,\ \ \nabla_{E_2}E_3=2E_1,\ \ \nabla_{E_3}E_3=0, \ \ \nabla_{E_1}E_2=\frac{2}{x}E_1\nonumber\\
\nabla_{E_2}E_1&=-2E_3,\ \ \nabla_{E_2}E_2=0,\ \ \nabla_{E_3}E_2=0,\ \ \nabla_{E_1}E_1=-\frac{2}{x}E_2.
\end{align}
From the above it can be easily seen that $(\phi, \xi, \eta, g)$ is a $N(k)$-contact metric
manifold with $k=-\frac{4}{x}\neq 0$.
\end{example}

Now, we are going to show that the above example is not correct. By using (\ref{hasan}),  we get
\begin{align}\label{hasan1}
R(E_1, E_2)E_3&=\nabla_{E_1}\nabla_{E_2}E_3-\nabla_{E_2}\nabla_{E_1}E_3-\nabla_{[E_1, E_2]}E_3\nonumber\\
&=2\nabla_{E_1}E_1-2\nabla_{E_3}E_3-\frac{2}{x}\nabla_{E_1}E_3\nonumber\\
&=-\frac{4}{x}E_2.
\end{align}
On the other hand,  since $\eta(E_1)=g(E_1, E_3)=0$ and $\eta(E_2)=g(E_2, E_3)=0$, then we have
\begin{equation}\label{hasan2}
R(E_1, E_2)E_3=k\big(\eta(E_2)E_1-\eta(E_1)E_2\big)=0,
\end{equation}
But (\ref{hasan1}) contradicts with (\ref{hasan2}).

\bigskip

In \cite{DG}, the authors proved the following theorem (See Theorem 4.1 in \cite{DG}).
\begin{theorem}{\rm (\cite{DG})}\label{3.2}
\emph{Every 3-dimensional $\phi$-recurrent $N(k)$-contact metric manifold is of constant curvature.}
\end{theorem}

On the other hand,  Blair proved the following.
\begin{theorem}{\rm (\cite{B})}\label{3.3}
\emph{A contact metric manifold $M^{2n+1}$ satisfying $R(X, Y)\xi=0$
is locally isometric to $E^{n+1}\times S^n(4)$ for $n>1$ and flat for $n=1$.}
\end{theorem}
\begin{rem}\label{rem1}
Using the above theorem, if $k=0$, then $N(k)$-contact metric manifold $M^3$ is flat. Thus according to Remark \ref{rem}, it is easy to see that   3-dimensional $N(k)$-contact metric manifold $M$ is $\phi$-recurrent, symmetric and locally $\phi$-symmetric. Therefore according to the Example \ref{Ex}, it is deduced that the Theorem \ref{3.2} is proved for $k\neq 0$.
\end{rem}

\bigskip

Here, we show that Theorem \ref{3.2} (Theorem 4.1 in \cite{DG}) is not correct. At first, we prove the following theorem:
\begin{theorem}\label{3.4}
Every  3-dimensional manifold $M$ is a $\phi$-recurrent $N(k)$-contact metric
manifold if and only if it is a flat manifold.
\end{theorem}
\begin{proof}
It is known that the Riemannian curvature of a 3-dimensional Riemannian manifold $M$ satisfies in
\begin{align}\label{N}
R(X, Y)Z&=g(Y, Z)QX-g(X, Z)QY+S(Y, Z)X-S(X, Z)Y\nonumber\\
&\ \ \ +\frac{r}{2}[g(X, Z)Y-g(Y, Z)X],
\end{align}
where $Q$ is the Ricci operator, that is , $g(QX, Y)=S(X, Y)$ and $r$ is the scalar curvature of $M$.
Now, let $M$ be a 3-dimensional $\phi$-recurrent $N(k)$-contact metric manifold. Putting $Z=\xi$ in (\ref{N}) and using (ii) of (\ref{esi3}) and $\eta(\xi)=1$, yields
\be
R(X, Y)\xi=(2k-\frac{r}{2})[\eta(Y)X-\eta(X)Y]+\eta(Y)QX-\eta(X)QY.\label{o1}
\ee
Part (i) of (\ref{esi3}) and (\ref{o1}) give us
\be
(k-\frac{r}{2})[\eta(Y)X-\eta(X)Y]=\eta(X)QY-\eta(Y)QX.\label{o2}
\ee
By setting $Y=\xi$ in (\ref{o2}) and using (ii) of (\ref{esi3}), it follows that
\begin{equation}\label{Na1}
QX=(\frac{r}{2}-k)X+(3k-\frac{r}{2})\eta(X)\xi,
\end{equation}
which gives us
\begin{equation}\label{Na2}
S(X, Y)=g(QX, Y)=(\frac{r}{2}-k)g(X, Y)+(3k-\frac{r}{2})\eta(X)\eta(Y).
\end{equation}
Using (\ref{Na1}), (\ref{Na2}) and (\ref{N}) we get
\begin{align}
R(X, Y)Z&=(3k-\frac{r}{2})[g(Y, Z)\eta(X)\xi-g(X, Z)\eta(Y)\xi+\eta(Y)\eta(Z)X\nonumber\\
&\ \ \ -\eta(X)\eta(Z)Y]+(\frac{r}{2}-2k)[g(Y, Z)X-g(X, Z)Y].\label{o3}
\end{align}
By (\ref{o3}), we get
\begin{align}\label{Im0}
(\nabla_WR)(X, Y)Z&=\nabla_WR(X, Y)Z-R(\nabla_WX, Y)Z\nonumber\\
&\ \ -R(X, \nabla_WY)Z-R(X, Y)\nabla_WZ\nonumber\\
&=\frac{dr(W)}{2}[g(Y, Z)X - g(X, Z)Y - g(Y, Z)\eta(X)\xi\nonumber\\
&\ \ + g(X, Z)\eta(Y)\xi - \eta(Y)\eta(Z)X+\eta(X)\eta(Z)Y]\nonumber\\
&\ \ + (3k-\frac{r}{2})[g(Y, Z)\eta(X) - g(X, Z)\eta(Y)]\nabla_W\xi\nonumber\\
&\ \ + (3k-\frac{r}{2})[\eta(Y)X - \eta(X)Y](\nabla_W\eta)(Z)\nonumber\\
&\ \ + (3k-\frac{r}{2})[g(Y, Z)\xi - \eta(Z)Y](\nabla_W\eta)(X)\nonumber\\
&\ \ - (3k-\frac{r}{2})[g(X, Z)\xi - \eta(Z)X](\nabla_W\eta)(Y).
\end{align}
Now, let $Y$ be a non-zero vector field orthogonal to $\xi$ and $X=Z=\xi$. Then from (\ref{Im0}), we get
\begin{equation}\label{sa}
(\nabla_WR)(\xi, Y)\xi=-2(3k-\frac{r}{2})(\nabla_W\eta)(\xi)Y.
\end{equation}
Since $\eta(\xi)=1$ and $\eta\circ\phi=0$, then using part (iii) of (\ref{Kill}) we obtain
\begin{equation}\label{sa1}
(\nabla_W\eta)(\xi)=W(\eta(\xi))-\eta(\nabla_W\xi)=\eta(\phi W+\phi hW)=0.
\end{equation}
By plugging (\ref{sa1}) in (\ref{sa}),  we have
\begin{equation}\label{sa2}
(\nabla_WR)(\xi, Y)\xi=0.
\end{equation}
Since $M$ is a $\phi$-recurrent manifold, then  there exists a non-zero 1-form $A$ such that satisfies in (\ref{rec}). Thus using (\ref{rec}) and (\ref{sa2}),  we deduce that
\begin{equation}\label{sa3}
A(W)R(\xi, Y)\xi=0.
\end{equation}
As $M$ is $N(k)$-contact metric manifold, then we have
\begin{equation}\label{sa4}
R(\xi, Y)\xi=k[\eta(Y)\xi-\eta(\xi)Y]=-kY.
\end{equation}
Setting (\ref{sa4}) in (\ref{sa3}), implies that $kA(W)Y=0$ which gives us $k=0$. Thus we have $R(X, Y)\xi=0$. Therefore, by using the  Theorem \ref{3.3} we can conclude  that $M^3$ is a flat manifold. From Remark \ref{rem1}, the converse of the theorem is obvious.
\end{proof}
Using Remark \ref{rem1} and the Theorem \ref{3.4},  we have the following.
\begin{cor}\label{3.6}
There is no $\phi$-recurrent $N(k)$-contact metric manifold with dimension 3, for $k\neq 0$.
\end{cor}
From the above corollary we deduce
\begin{cor}
There is no 3-dimensional $\phi$-recurrent Sasakian manifold.
\end{cor}

Using Remark \ref{rem1}, Theorem \ref{3.4} and Corollary \ref{3.6},  we conclude the following.
\begin{theorem}\label{3.7}
There is no 3-dimensional $\phi$-recurrent $N(k)$-contact metric
manifold which is neither symmetric nor locally $\phi$-symmetric
\end{theorem}
Therefore, by using Remark \ref{rem1}, Corollary \ref{3.6} and Theorem \ref{3.7}, we conclude that Theorem \ref{3.2} is not correct.
\section{$\phi$-recurrent contact metric  manifolds of constant curvature}
In this section, we show that the only flat manifolds of dimension 3 can be $\phi$-recurrent contact metric manifolds of constant curvature. For this work,  we present  a fact that shows  Theorem \ref{3.2} (Theorem 4.1 in \cite{DG}) is not correct, again.
\begin{theorem}\label{Fardin}
For $k\neq 0$, there exists no  $(2n+1)$-dimensional $\phi$-recurrent $N(k)$-contact metric manifold of constant curvature.
\end{theorem}
\begin{proof}
Let $M^{2n+1}$ be a $\phi$-recurrent $N(k)$-contact metric manifold which has the constant curvature $\lambda$. Then we have
\begin{equation}\label{curv}
R(X, Y)Z=\lambda(g(X, Z)Y-g(Y, Z)X).
\end{equation}
Setting $Z=\xi$ in (\ref{curv}) yields
\begin{equation}\label{curv1}
R(X, Y)\xi=\lambda(\eta(X)Y-\eta(Y)X).
\end{equation}
Since $M$ is $N(k)$-contact metric manifold, then we have
\begin{equation}\label{curv2}
R(X, Y)\xi=k(\eta(Y)X-\eta(X)Y).
\end{equation}
(\ref{curv1}) and (\ref{curv2}) give us
\be
k(\eta(Y)X-\eta(X)Y)=\lambda(\eta(X)Y-\eta(Y)X).\label{o5}
\ee
Let $X=\xi$ and $Y$ be a non-zero vector field orthogonal to $\xi$. Then we have $\eta(Y)=0$ and $\eta(X)=1$. Thus from (\ref{o5}),  we deduce that
\[
-kY=\lambda Y, \ \ \  \textrm{or} \ \ \  \lambda=-k\neq 0.
\]
Thus by using (\ref{curv}),  we have
\begin{equation}\label{curv3}
R(X, Y)Z=-k(g(X, Z)Y-g(Y, Z)X).
\end{equation}
By (\ref{curv3}), it follows that
\begin{align}\label{Im}
(\nabla_WR)(X, Y)Z&=\nabla_WR(X, Y)Z-R(\nabla_WX, Y)Z-R(X, \nabla_WY)Z\nonumber\\
&\ \ -R(X, Y)\nabla_WZ=-k[(\nabla_Wg(X, Z))Y+g(X, Z)\nabla_WY\nonumber\\
&\ \ -(\nabla_Wg(Y, Z))X-g(Y, Z)\nabla_WX-g(\nabla_WX, Z)Y\nonumber\\
&\ \ +g(Y, Z)\nabla_WX-g(X, Z)\nabla_WY+g(\nabla_WY, Z)X\nonumber\\
&\ \ -g(X, \nabla_WZ)Y+g(Y, \nabla_WZ)X]\nonumber\\
&=-k[((\nabla_Wg)(X, Z))Y-((\nabla_Wg)(Y, Z))X]\nonumber\\
&=0.
\end{align}
Also from (\ref{curv3}),  we get
\begin{equation}\label{esi10}
R(X, \xi)\xi=k[X-\eta(X)\xi].
\end{equation}
Putting $Y=Z=\xi$ in (\ref{rec}) and using (\ref{Im}) and (\ref{esi10}) imply  that
\be
kA(W)[X-\eta(X)\xi]=0.\label{aki}
\ee
If $X$ is a non-zero vector field orthogonal to $\xi$, then (\ref{aki}) gives us $kA(W)X=0$, which is a contradiction to $k\neq 0$ and $A(W)\neq 0$.
\end{proof}

In \cite{B}, Blair proved the following.
\begin{theorem}\label{Blair}
If a contact metric manifold $M^{2n+1}$ is of constant curvature
$\lambda$ and $n>1$, then $\lambda=1$ and the structure is Sasakian.
\end{theorem}

Here,  we are going to  consider the same result for  $\phi$-recurrent contact metric manifold. Then we prove the following.
\begin{theorem}\label{T}
there exists no  $(2n+1)$-dimensional $\phi$-recurrent contact metric manifold of constant curvature in which $n>1$.
\end{theorem}
\begin{proof}
Let $M$ be a $\phi$-recurrent contact metric manifold of constant curvature $\lambda$ with dimension $2n+1$, where $n>1$. Considering Theorem \ref{Blair}, we deduce that $\lambda=1$. Thus we have
\begin{equation}\label{F}
R(X, Y)Z=g(X, Z)Y-g(Y, Z)X.
\end{equation}
Similar to proof of (\ref{Im}) in Theorem \ref{Fardin}, by (\ref{F}) we get the following
\begin{equation}\label{F1}
(\nabla_WR)(X, Y)Z=0.
\end{equation}
Let $X=Z=\xi$ and $Y$ be a non-zero vector field orthogonal to $\xi$. Then (\ref{F}) gives us
\begin{equation}\label{F2}
R(\xi, Y)\xi=Y.
\end{equation}
Since $M$ is $\phi$-recurrent, then by using (\ref{rec}), (\ref{F1}) and (\ref{F2}) we deduce that
\[
A(W)Y=0.
\]
But this contradicts $A(W)\neq 0$.
\end{proof}

\bigskip

\begin{theorem}\label{4.4}
A $\phi$-recurrent contact manifold $M^{3}$ is of constant curvature $\lambda$ if and only if $\lambda=0$.
\end{theorem}
\begin{proof}
Let $M$ be a 3-dimensional $\phi$-recurrent contact manifold of constant curvature $\lambda$. Then we have
\[
R(X, Y)Z=\lambda(g(X, Z)Y-g(Y, Z)X).
\]
Since $M^3$ is $\phi$-recurrent then similar to proof of Theorem \ref{T}, by using the above equation,  we obtain $\lambda A(W)Y=0$, where $A$ is a non zero 1-form and $Y$ is a non zero vector field on $M$. Thus we deduce that $\lambda=0$. According to Remarks \ref{rem} and \ref{rem1}, the converse of the theorem is obvious.
\end{proof}

\bigskip

Theorems \ref{T} and \ref{4.4} give us the following.
\begin{cor}
The only flat 3-dimensional $\phi$-recurrent contact metric manifolds,  are manifolds of constant curvature.
\end{cor}

\bigskip

\noindent
Esmaeil Peyghan and Hassan Nasrabadi\\
Department of Mathematics, Faculty  of Science\\
Arak University\\
Arak 38156-8-8349,  Iran\\
Email: epeyghan@gmail.com
\bigskip

\noindent
Akbar Tayebi\\
Department of Mathematics, Faculty  of Science\\
University of Qom \\
Qom. Iran\\
Email:\ akbar.tayebi@gmail.com

\end{document}